\documentclass{article}
\usepackage{amscd,amsmath}
\usepackage[dvips]{graphicx}
\usepackage{hyperref}
\usepackage{array}
\usepackage[all]{xy}
\usepackage{amsfonts}
\usepackage{amsmath}

\newtheorem{theorem}{\bf Theorem}[section]
\newtheorem{proposition}{\bf Proposition}[section]

\newtheorem{definition}{\bf Definition}[section]
\newtheorem{remark}{\bf Remark}[section]

\newenvironment{proof}[1][Proof.]{\begin{trivlist} \item[\hskip \labelsep {\bfseries #1}]}{\end{trivlist}}
\title{Linear Cocycles over Lorenz-like Flows}
\author{Mohammad Fanaee}

\begin{document}

\maketitle

\begin{abstract}
We prove that the Lyapunov exponents of typical fiber bunched linear cocycles over Lorenz-like flows have multiplicity one: the set of exceptional cocycles has infinite codimention, i.e. it is locally contained in finite unions of closed submanifolds with arbitrarily high codimension.\\
\end{abstract}

\setcounter{tocdepth}{1}

\tableofcontents

\section{Introduction}
A linear cocycle over a flow $f^t:\Lambda\rightarrow\Lambda$ is a flow $F^t:\Lambda\times\mathbb C^d\rightarrow\Lambda\times\mathbb C^d$ of the form
$$F^t(x,v)=(f^t(x),A^t(x)v)$$
where each $A^t(x):\mathbb C^d\rightarrow\mathbb C^d$ is a linear isomorphism. The Lyapunov exponents are the exponential rates
$$\lambda(x,v)=\lim_{|t|\rightarrow\infty}\frac{1}{t}\log||A^t(x)v||,~v\neq0.$$

By Oseledets [14] this limit exists for every $v\in\mathbb C^d$ on a full measure set of $x\in\Lambda$, relative to any invariant measure $m$. There are at most $d$ Lyapunov exponents; they are constant on orbits and vary measurably with the base point. Thus Lyapunov exponents are constant if $m$ is ergodic.\\

Our main interest is to characterize when all exponents have multiplicity one i.e. the subspace of vectors $v\in\mathbb C^d$ that share the same value of $\lambda(x,v)$ has dimension one.

There has been much recent progress on this problem, specially when the base dynamics is hyperbolic, see [9,5,6,11]. Here, we extend the theory to the case when the base dynamics is a Lorenz-like attractor.\\

A Lorenz-like flow in 3-dimensions admits a cross section $S$ and a Poincar\'e transformation $P:S\backslash\Gamma\rightarrow S$ defined outside a curve $\Gamma$ which is contained in the intersection of $S$ with the stable manifold of some hyperbolic equilibrium. Trajectories through $\Gamma$ just converge to the equilibrium and the other trajectories through $S$ eventually return to $S$. Their accumulation set is the so-called geometric Lorenz attractor. Moreover, there is an invariant splitting
$$T_{\Lambda}M=E^s\oplus E^{cu}$$
of the tangent bundle where the uniformly contracting bundle $E^s$ has dimension 1, and the volume-expanding bundle $E^{cu}$ which contains the flow direction has dimension 2. Another important feature is that the Poincar\'e transformation of this flow admits an invariant contracting foliation $\mathcal F$ through which the dynamics can be reduced to that of a map on the interval (leaf space of $\mathcal F$). A Lorenz-like flow admites an invariant physical probability measure which is ergodic.

\subsection{Cocycles over maps}
A linear cocycle over an invertible transformation $f:N\rightarrow N$  is a transformation $F:N\times\mathbb C^d\rightarrow N\times\mathbb C^d$ satisfying $f\circ \pi=\pi\circ F$ which acts by linear isomorphisms $A(x)$ on fibers. So, the cocycle has the form
$$F(x,v)=(f(x),A(x)v)$$
where
$$A:N\rightarrow\mathrm{GL}(d,\mathbb C).$$
Conversely, any $A:N\rightarrow\mathrm{GL}(d,\mathbb C)$ defines a linear cocycle over $f$. Note that $F^n(x,v)=(f^n(x),A^n(x)v)$, where
$$A^n(x)=A(f^{n-1}(x))~...~A(f(x))A(x),$$
$$A^{-n}(x)=(A^n(f^{-n}(x)))^{-1},$$
for any $n\geq 1$, and $A^0(x)=\mathrm{id}$.

Let $\mu$ be a probability measure invariant by $f$. Oseledets Theorem [14] states that there exist a Lyapunov splitting
 $$E_1(x)\oplus...\oplus E_k(x),~1\leq k=k(x)\leq d,$$
and Lyapunov exponents $\lambda_1(x)>...>\lambda_k(x)$,
$$\lambda_i(x)=\lim_{|n|\rightarrow\infty}\frac{1}{n}\log\parallel A^n(x)v_i\parallel,~v_i\in E_i(x),~1\leq i\leq k,$$
at $\mu$-almost every point. Lyapunov exponents are invariant, uniquely defined at almost every $x$ and vary measurably with the base point $x$. Thus, Lyapunov exponents are constant when $\mu$ is ergodic. Then $\{\lambda_1,...,\lambda_k\}$ is  called the Lyapunov spectrum of $A$.

We recall that, for any $r\in\mathbb N\cup\{0\}$ and $0\leq\rho\leq 1$, the $C^{r,\rho}$ topology is defined by
$$||A||_{r,\rho}=\max_{0\leq i\leq r}\sup_x||D^iA(x)||+\sup_{x\neq y}\dfrac{||D^rA(x)-D^rA(y)||}{\mathrm{d}(x,y)^{\rho}}$$
(for $\rho=0$ omit the last term) and then
$$\mathcal C^{r,\rho}(N,d,\mathbb C)=\{A:N\rightarrow\mathrm{GL}(d,\mathbb C):~||A||_{r,\rho}<+\infty\}$$
is a Banach space. We assume that $r+\rho>0$ which implies $\eta-$H\"older continuity:
$$\parallel A(x)-A(y)\parallel\leq\parallel A\parallel_{0,\eta}\mathrm{d}(x,y)^\eta,$$
with
\[\eta=\left\{\begin{array}{cl}
\rho&r=0\\1&r\geq 1.
\end{array}\right.\]

\subsection{Fiber bunching condition}
Suppose that $N=\mathbb N^{\mathbb Z}$, the full shift space with countably many symbols, and  $f:N\rightarrow N$ the shift map
$$f((x_n)_{n\in\mathbb Z})=(x_{n+1})_{n\in\mathbb Z}.$$
A cylinder of $N$ is any subset
$$[a_k,...;a_0;...,a_l]=\{x:~x_j=\iota_j,~j=k,...,l\}$$
of $N$.
We endowed $N$ with topology generated by cylinders. The local stable and local unstable sets of any $x\in N$ are defined as
$$W^s_{\mathrm{loc}}(x)=\{y:~x_n=y_n,~n\geq 0\}$$
and
$$W^u_{\mathrm{loc}}(x)=\{y:~x_n=y_n,~n<0\}.$$

Assume that $N$ is endowed with a metric d for which\\
(i) $\mathrm d(f(y),f(z))\leq\theta(x)\mathrm d(y,z)$, for all $y,z\in W^s_{\mathrm{loc}}(x)$,\\
(ii) $\mathrm d(f^{-1}(y),f^{-1}(z))\leq\theta(x)\mathrm d(y,z)$, for all $y,z\in W^u_{\mathrm{loc}}(x)$,\\
where $0<\theta(x)\leq\theta<1$, for all $x\in N$.

Let $A$ be an  $\eta$-H\"{o}lder continuous linear cocycle over $f$.

\begin{definition}
$A$ is fiber bunched if there exists some constant $\tau\in(0,1)$ such that
$$||A(x)||~||A(x)^{-1}||\theta(x)^{\eta}<\tau,$$
for any $x\in N$.
\end{definition}

\begin{remark}
Fiber bunching is an open condition in $C^{r,\rho}(N,d,\mathbb C)$: if $A$ is a fiber bunched linear cocycle then any linear cocycle $B$ sufficiently $C^0$ close to $A$ is also fiber bunched, by definition.
\end{remark}

\subsection{Product structure}
Let $N_u=\mathbb N^{\{n\geq 0\}}$ and $N_s=\mathbb N^{\{n<0\}}$. The map
$$x\mapsto(x_s,x_u)$$
is a homeomorphism form $N$ onto $N_s\times N_u$ where $x_s=\pi_s(x)$ and $x_u=\pi_u(x)$, for natural projections $\pi_s:N\rightarrow N_s$ and $\pi_u:N\rightarrow N_u$. We also consider the maps $f_s:N_s\rightarrow N_s$ and $f_u:N_u\rightarrow N_u$ defined by
$$f_u\circ\pi_u=\pi_u\circ f,$$
$$f_s\circ\pi_s=\pi_s\circ f^{-1}.$$

Assume that $\mu_f$ is an ergodic probability measure for $f$. Let $\mu_s=(\pi_s)_*\mu_f$ and $\mu_u=(\pi_u)_*\mu_f$ be the images of $\mu_f$ under the natural projections. It is easy to see that $\mu_s$ and $\mu_u$ are ergodic probabilities for $f_s$ and $f_u$, respectively. Notice that $\mu_s$ and $\mu_u$ are positive on cylinders, by definition.

We say that $\mu_f$ has product structure if there exists a measurable density function $\omega:N\rightarrow(0,+\infty)$ such that
$$\mu_f=\omega(x)(\mu_s\times\mu_u).$$

Assuming a probability measure which has product strucrure, Bonatti and Viana [9] obtained a general criterion for simplicity of Lyapunov spectrum for cocycles over hyperbolic systems and used it to prove that simplicity holds for generic linear cocycles that satisfy the fiber bunching condition. This criterion has improved by Avila and Viana [5] who used it to prove the Zorich-Kontsevich conjecture [6]. In [11], by geometric tools, we prove

\begin{theorem}
$\mathrm{[F]}$ Lyapunov exponents of typical fiber bunched linear cocycles over complete shift map have multiplicity 1.
\end{theorem}

\subsection{Suspension flows}
Consider a suspension flow $f^t:\Lambda\rightarrow\Lambda$ of $f:N\rightarrow N$ and let $T:N\rightarrow\mathbb R$ be the corresponding return time to $N$. Assume that $A^t:\Lambda\rightarrow\mathrm{GL}(d,\mathbb C)$ is a linear cocycle over $f^t$, and define
$$A_f(x)=A^{T(x)}(x),$$
for any $x\in N$. Note that $A_f:N\rightarrow\mathrm{GL}(d,\mathbb C)$ is a linear cocycle over $f$.

Then we define a relative topology as
$$||A^t||_{r,\rho}=||A_f||_{r,\rho}$$
for any $r\in\mathbb N\cup\{0\}$ and $0\leq\rho\leq 1$ with $r+\rho>0$, and let
$$\mathcal C^{r,\rho}(\Lambda,d,\mathbb C)=\{A^t:\Lambda\rightarrow\mathrm{GL}(d,\mathbb C):~||A^t||_{r,\rho}<+\infty\}.$$

\begin{definition}
$A^t$ is fiber bunched if the corresponding linear cocycle $A_f$ is a fiber bunched linear cocycle over $f$.
\end{definition}

\begin{remark}
Note that fiber bunching is an open condition in $\mathcal C^{r,\rho}(\Lambda,d,\mathbb C)$, by definition.
\end{remark}

Our main result is\\\\
\textbf{Main Theorem.}  Typical fiber bunched linear cocycles over geometric Lorenz attractors have simple spectrum.

\section{Lorenz-like flows}
In this section, we recall the basic notions and strategies to construct a geometric Lorenz attractor and the unique physical probability measure and then, we study existence of a Markov structure on these flows.

The geometric Lorenz attractors were introduced in [18,12] as a precise model for the dynamical behavior of the equations

\begin{eqnarray}
\begin{array}{l} \dot x=a(y-x), \\
\dot y=bx-y-xz, \\
\dot z=xy-cx,
\end{array}
\end{eqnarray}
proposed by Lorenz [13],  loosely related to fluid convection and weather prediction. Tucker [16] showed that the Lorenz equations exhibits a geometric Lorenz attractor, for classical parameters $a=10, ~ b=28, ~ c=8/3$.\\

This system of equations is symmetric with respect to the $z-$axis. The singularity \textbf{0} has real eigenvalues $\alpha_{ss}<\alpha_s<0<-\alpha_{ss}<\alpha_u$ with $\alpha_s+\alpha_u>0$. There are also two symmetric saddles $\sigma_1,\sigma_2$ with a real negative and two conjugate complex eigenvalues where the complex eigenvalues have positive real parts. The character of this flow is strongly dissipative, in particular, any maximally positively invariant subset has zero volume.

\subsection{The geometric model}
To construct a geometric Lorenz attractor, we should analyze the dynamics of Lorenz flow in a neighborhood of \textbf{0} imitating the effect of the pair of saddles.

\subsubsection{Poincar\'e transformation}
By construction, there is a cross section $S$ intersecting the stable manifold of $0$ along a curve $\Gamma$ that separates $S$ into $2$ connected components. We denote the corresponding Poincar\'e transformation $$P:S\backslash\Gamma\rightarrow S.$$
Note that the future trajectories of points in $\Gamma$ do not came back to $S$.

We consider the smooth foliation $\mathcal F$ of $S$ into curves having $\Gamma$ as a leaf which are invariant and uniformly contracted by forward iterates of $P$. Indeed, every leaf $\mathcal F_{(x,y)}$ is mapped by $P$ completely inside the leaf $\mathcal F_{P(x,y)}$, and $P|_{\mathcal F_{(x,y)}}$ is a uniform contraction. Indeed, $P$ must have the form
$$P(x,y)=(g(x),h(x,y))$$
which by effect of saddles and singularity, we can assume that $h$ is a contraction along its second coordinate. The map $g$ is uniformly expanding with derivative tending to infinity as one approaches to $\Gamma$. We assume that $|g^\prime|\geq\theta^{-1}>\sqrt 2$ and since the rate of contraction of $h$ on the second coordinate should be much higher than the expansion of $g$, we can take $|\partial_yh|\leq \theta<1$.

\subsubsection{Lorenz map}
Let $\pi$ be the canonical projection of section $S$ into $\mathcal F$, i.e. $\pi$ assigns to each point of $S$ the leaf that contained it. By invariance of $\mathcal F$, one dimensional Lorenz map
$$g:(\mathcal F\backslash\Gamma)\rightarrow\mathcal F$$
is uniquely defined so that
\[\begin{CD}
S\backslash\Gamma @>P>> S\\
@V\pi VV @VV\pi V\\
\mathcal F\backslash\Gamma @>>g>\mathcal F
\end{CD}\]\\
commutes, i.e. $g\circ\pi=\pi\circ P$ on $S\backslash\Gamma$.

One may identify quotient space $S/\mathcal F$ with a compact interval as $I=[-1,1]$, and so
$$g:[-1,1]\backslash\{0\}\rightarrow[-1,1]$$
is smooth on $I\backslash\{0\}$ with a discontinuity and infinite left and right derivatives at $0$. Note that the symmetry of the Lorenz equations implies $g(-x)=-g(x)$.

\subsection{The attractor}
The geometric Lorenz attractor $\Lambda$ is characterized as follows. Note that the restriction of $g$ to both $\{x<0\}$ and $\{x>0\}$ admits continuous extensions to the point $0$. Hence, $g$ may be considered as an extension to a 2-valued map at $0$ and continuous on both $\{x\leq 0\}$ and $\{x\geq 0\}$. Correspondingly, the restriction of the Poincar\'e transformation to each of the connected components of $S\backslash\Gamma$ admits a continuous extension to the closure, each one collapsing the curve $\Gamma$ to a single point. Thus, $P$ may also be considered as a 2-valued transformation defined on the whole cross section and continuous on the closure of each of the connected components. Let
$$\Lambda_P=\bigcap_{n\geq 0}P^n(S)\subset S.$$
We define $\Lambda$ to be the saturation of $\Lambda_P$ by the Lorenz flow, that is, the orbits of its points. Therefore, orbits in $\Lambda$ intersect the cross section infinitely often, both forward and backward.

Dynamical properties of $\Lambda$ may be deduced from corresponding properties for the quotient map $h$. More important, a quotient map with similar properties exists for all nearby vector fields, and so such properties are robust for these flows.

\subsection{Physical probability measure}
The existence of a unique absolutely continuous invariant probability $\mu_g$ which is ergodic and $0<\frac{d\mu_g}{d(Leb)}<+\infty$ for Lorenz one-dimensional map $g$ is well-known (see [16] for more details).

One may construct an invariant probability measure $\mu_P$ on $\Lambda_P$, as the lifting of $\mu_g$. Indeed, we may think of $\mu_g$ as a probability measure on Borel subsets of $\mathcal F$. Since $P$ is uniformly contracting on leaves of $\mathcal F$, one concludes that the sequence
$$(P^n_*\mu_g)_{n\geq 1},$$
of push-forwards is weak$^*$-Cauchy: given any continuous $\varphi:S\rightarrow\mathbb R$,
$$\int\varphi~d(P^n_*\mu_g)=\int(\varphi\circ P^n)~d\mu_g,~n\geq 1,$$
is a Cauchy sequence in $\mathbb R$. Define the probability measure $\mu_P$ as the weak$^*$-limit of this sequence that is
$$\int\varphi~d\mu_P=\lim_{n\rightarrow+\infty}\int\varphi~d(P^n_*\mu_g),$$
for each continuous function $\varphi$. Thus $\mu_P$ is invariant under $P$, and it is a physical probability measure on Borel subsets of $\Lambda_P$ which is ergodic.

Later, as the Poincar\'e transformation may be extended to the Lorenz flow through a suspension construction, the invariant probability $\mu_P$ corresponds to an ergodic physical probability measure $m$ on $\Lambda$: Denote by
$R:S\backslash\Gamma\rightarrow(0,+\infty)$ the \textit{first return time} to $S$ defined by
$$P(x)=f^{R(x)}(x).$$
The first return time $R$ is Lebesgue integrable, since
$P(x)\approx|\log(\mathrm{d}(x,\Gamma))|$, for $x$ close to $\Gamma$. This follows that
$$\int R~d\mu_P<+\infty.$$
Let $\sim$ be an equivalence relation on $S\times\mathbb R$ defined as $(x,R(x))\sim(P(x),0)$. Set $\tilde S=(S\times\mathbb R)/\sim$ and define the finite measure
$$\tilde\mu=\pi_*(\mu_P\times dt)$$
where $\pi:S\times\mathbb R\rightarrow \tilde S$ is the quotient map and $dt$ is Lebesgue measure in $\mathbb R$. Define $\phi:\tilde S\rightarrow M$ as $\phi(x,t)=f^t(x)$, and let
$$m=\phi_*\tilde\mu.$$
One may check also that
$$\frac{1}{T}\int_0^T\varphi(f^t(x))~dt~\rightarrow~\int\varphi~dm$$
as $T\rightarrow+\infty$, for any continuous function $\varphi:M\rightarrow\mathbb R$, and Lebesgue almost every $x\in\phi(\tilde S)$.

\section{A symbolic structure}
Consider a Lorenz one dimensional map $g:I\backslash\{0\}\rightarrow I$.
\begin{theorem}
$\mathrm{[10]}$ There exists a return map $\hat g$, an interval $\hat I=(-\delta,\delta),~0<\delta<1$, and a partition $\{\hat I(l):~l\in\mathbb N\}$ to subintervals of $\hat I$, Lebesgue mod 0, for which $\hat g$ maps any $\hat I(l)$ diffeomorphically onto $\hat I$, and the return time $\hat r$ is Lebesgue integrable. Moreover, there exists a constant $0<c<1$ such that, for all $x,y$ in any $\hat I(l)$,
$$\log\frac{|\hat g^\prime(x)|}{|\hat g^\prime(y)|}\leq c^{n(x,y)}$$
where $n(x,y)=\min\{n:~\hat g^n(x)\in\hat I(l_i),~\hat g^n(y)\in\hat I(l_j),~i\neq j\}$.\\
\end{theorem}

\begin{remark}
Note that, as Lorenz map $g$ is uniformly expanding, the intersection of $(\hat g^{-n}(J(l_n)))$ over all $n\geq 0$ consists of exactly one point.
\end{remark}

Therefore, $\hat g$ may be seen as the shift map on $\hat N=\mathbb N^{\{n\geq 0\}}$: there exists a conjugation between the shift map $\hat f:\hat N\rightarrow\hat N$ and $\hat g$ presented by the next commuting diagram\\
\[\begin{CD}  
\hat N @>\hat f>>\hat N\\
@V\hat\phi VV @VV\hat\phi V\\
\hat I @>>\hat g>\hat I\\
\end{CD}\]\\
where the bijection $\hat\phi$ may be defined as
$$\hat\phi:(l_n)_{n\geq 0}\mapsto\bigcap_{n\geq 0}\hat g^{-n}(\hat I(l_n)).$$

\subsection{Bi-dimensional Markov structure}
Now, we consider the bi-dimensional domain $\hat S=\pi^{-1}(\hat I)\subset S$ and corresponding to the Markov partition of $\hat I$ define a Markov partition $\{\hat S(l)=\pi^{-1}(\hat I(l)):~l\in\mathbb N\}$ of $\hat S$. The return time is defined as
$$r(x)=\hat r(\pi(x)).$$

Hence, there exists a return map $\hat P$ to $\hat S$ as
$$\hat P(x)=P^{r(x)}(x),$$
for any $x\in\hat S$. Moreover
$$\hat g\circ\pi=\pi\circ\hat P.$$

Let
$$\Lambda_{\hat P}=\bigcap_{n\geq 0}\hat P^n(\hat S).$$
So $\Lambda_{\hat P}$ is homeomorphically equal to $N$. Indeed, since $\bigcap_{n\in\mathbb Z}\hat P^{-n}(\hat S(l_n))$ consists of exactly one point, one may define a bijection $\phi:N\rightarrow\Lambda_{\hat P}$ as
$$\phi:(l_n)_{n\in\mathbb Z}\mapsto\bigcap_{n\in\mathbb Z}\hat P^{-n}(\hat S(l_n))$$
which implies the commuting diagram
\[\begin{CD}
N @>f>>N\\
@V\phi VV @VV\phi V\\
\Lambda_{\hat P} @>>\hat P> \Lambda_{\hat P}.
\end{CD}\]

\subsection{Lifting the probability measure}
The normalized restriction $\hat\mu$ of $\mu_g$ to the domain of $\hat g$ is an absolutely continuous ergodic probability for $\hat g$ and then for $\hat f$, by conjugacy.

As the natural extension of $\hat f$ realized as the complete shift map $f$ on $N$, the \textit{lift} $\mu$ of $\hat\mu$ is the unique $f$-invariant ergodic probability measure on $N$ such that
$$\hat\pi_*\mu=\hat\mu.$$

\begin{proposition}
The lift probability $\mu$ has product structure. Moreover, the density function $\omega$ is continuous and, bounded from zero and infinity
\end{proposition}

\begin{proof}
Note that by Theorem 3.1, for all $\hat x,\hat y$ in the same cylinder
$$\log\frac{J\hat f(\hat x)}{J\hat f(\hat y)}\leq c^{n(x,y)}.$$
The rest of proof is based on 4 main steps
\textit{Step 1.} If $\hat x,\hat y\in\hat N$ then for any $x\in W^s_{\mathrm{loc}}(\hat x)$ and $y\in W^u_{\mathrm{loc}}(x)\cap W^s_{\mathrm{loc}}(\hat y)$, the limit
$$J_{\hat x,\hat y}(x)=\lim_{n\rightarrow\infty}\frac{J\hat f^n(\hat x^n)}{J\hat f^n(\hat y^n)},$$
where $\hat x^n=\hat\pi(f^{-n}(x)),~\hat y^n=\hat\pi(f^{-n}(y))$, exists uniformly on $\hat x,\hat y,x$. Moreover,
$$(\hat x,\hat y,x)\mapsto J_{\hat x,\hat y}(x)$$
is continuous and uniformly bounded from zero and infinity.

Indeed, we observe that
$$\log\frac{J\hat f^n(\hat x^n)}{J\hat f(\hat y^n)}\leq\sum_{i=1}^n\log\frac{J\hat f(\hat x^i)}{J\hat f(\hat y^i)}.$$

Since $\hat x^i$ and $\hat y^i$ are in the same cylinder, the series is uniformly bounded by $\sum_ic^{n(\hat x^i,\hat y^i)}$. But $n(\hat x^i,\hat y^i)$ is strictly increasing that implies uniform convergence of the series.\\
\textit{Step 2.} If $\{\mu_{\hat x}:~\hat x\in\hat N\}$ be an integration of $\mu$ then, for $\mu$-almost every $\hat x\in\hat N$,
$$\mu_{\hat x}(\xi_n)=\frac{1}{J\hat f^n(\hat x^n)},$$
for every cylinder $\xi_n=[x_{-n},...,x_{-1}],~n\geq 1$, and any $x\in\xi_n\times\{\hat x\}$.\\
\textit{Step 3.} Given any disintegration, by the last step, one may find a disintegration $\{\mu_{\hat x}:~\hat x\in\hat N\}$ of $\mu$ so that
$$\mu_{\hat y}=J_{\hat x,\hat y}\mu_{\hat x}.$$
\textit{Step 4.} Fixing any $\hat x_0\in\hat N$, one may define
$$\hat\omega(x_s,x_u)=J_{\hat x_0,x_u}(x_s,x_u),$$
for every $x=(x_s,x_u)\in N$. By Step 2, $\mu_{x_u}=\hat\omega(x_s,x_u)$, for any $x_u\in\hat N$.

The lift measure $\mu$ projects to $\hat\mu=\mu_u$, but the projection $\mu_s$ to $N_s$ is given by
$$\mu_s=\mu_{\hat x_0}\int_{\hat N}\hat\omega(x_s,x_u)~d\hat\mu.$$

Therefore
$$\mu=\omega(x_s,x_u)\mu_s\times\mu_u$$
where
$$\omega(x_s,x_u)=\frac{1}{\int_{\hat N}\hat\omega(x_s,x_u)~d\hat\mu}~\hat\omega(x_s,x_u).$$

As conditional probabilities vary continuously with the base point so the density function $\omega$ is continuous. Also, $\omega$ is bounded from zero and infinity.

The i of Proposition 3.1 is now completed.\\
\end{proof}

\subsection{Suspending the bi-lateral shift}
The saturation of $N$ by the Lorenz flow $f^t$, by ergodicity of $m$ has full measure in $\Lambda$. Now on, by $\Lambda$ we mean this full measure subset. Henceforth, a return time to $N$ is defined as
$$T: N\rightarrow\mathbb R$$
$$T(x)=\sum_{j=0}^{r(x)-1}R(P^j(x)),$$
for any $x\in N$.

\begin{proposition}
The return time $T$ is integrable with respect to the probability measure $\mu$.
\end{proposition}
\begin{proof}
For almost every $x$,
$$\int T(x)~d(Leb)=\int r(x)[\frac{1}{r(x)}\sum_{j=0}^{r(x)-1}R(P^j(x))]~d(Leb)$$
converges to
$$\int r(x)(\int R~d(Leb))~d(Leb)<+\infty$$
which implies
$$\int T~d(Leb)<+\infty.$$

The proof is now completed by absolute continuity.
\end{proof}

\section{The proof of Main Theorem}
Now, we are in the setting to complete the proof of Main Theorem.\\

For any linear cocycle $A^t$ over $\Lambda$ consider the corresponding linear cocycle $A_f$ on $N$ by
$$A_f(x)=A^{T(x)}(x),$$
for any $x\in N$.

\begin{proposition}
Lyapunov spectrum of $A^t$ is simple if and only if Lyapunov spectrum of $A_f$ is simple.
\end{proposition}

\begin{proof}
The Lyapunov exponents of $A_f$ are obtained by multiplying those of $A^t$ by the average return time
$$~s_n(x)=\sum_{j=0}^{n-1} T(\hat P^j(x)),~x\in N.$$
Indeed, given any non zero vector $v$,
$$\lim_{n\rightarrow+\infty}\frac{1}{n}\log||A_f^n(x)v||=\lim_{n\rightarrow+\infty}\frac{1}{n}\log||A^{s_n(x)}(x)v||$$
which, for $\mu$-almost every $x$, this is equal to
$$\lim_{n\rightarrow+\infty}\frac{1}{n}s_n(x)\lim_{m\rightarrow+\infty}\frac{1}{m}\log||A^m(x)v||.$$
But $\frac{1}{n}s_n(x)$ converges to $\int T~d\mu<+\infty$

The proof of Proposition 4.1 is now completed.\\
\end{proof}

Let $A^t$ be a linear cocycle over $\Lambda$. We define a neighborhood $\mathcal V$ of $A^t$ as the subset of all cocycles $B^t$ over $\Lambda$ for which $B_f\in\mathcal U$.

\begin{proposition}
The application
$$\mathcal V\ni B^t\mapsto B_f\in\mathcal U$$
is a submersion.
\end{proposition}

\begin{proof}
By definition,
$$\partial_{B^t}B_f(\dot B_t)=\dot B_f.$$

Let $\dot B\in\mathcal C^{r,\rho}(N,d,\mathbb C)$. Then the suspension $\dot B^t$ of $\dot B$ is defined by
$$\dot B^t(X^s(x))=(\mathrm{id},t+s),~0<t+s\leq T(x),$$
identifying $(\mathrm{id},T(x))$ with $(\dot B(x),0)$, for any $x\in N$, setting $\dot B^0=\mathrm{id}$. $\dot B^t$ is an $\eta$-H\"older linear cocycle over $\Lambda$ for which $\dot B_f(x)=(\dot B(x),0)$. This shows that the derivative is surjective.

The proof of Proposition 4.2 is now completed.
\end{proof}

The proof of Main Theorem is then completed, by Theorem 1.1.\\\\
\textbf{\textit{Acknowledgments.}}
I would like to thanks M. Viana for all supports during my PhD studies at IMPA. This work is supported by a doctoral grant from CNPq-TWAS.

\textit{Mohammad Fanaee}

\textit{Instituto de Matemática e Estatística (IME)}

\textit{Universidade Federal Fluminense (UFF)}

\textit{Campus Valonguinho}

\textit{24020-140 Niter\'oi-RJ-Brazil}

\textit{Email: mf@id.uff.br}
\end{document}